\documentclass{amsart}

\numberwithin{equation}{section}

\usepackage{amssymb}
\usepackage{enumerate, xspace}
\usepackage{marvosym} 
\usepackage[sans]{dsfont}        
\usepackage[backgroundcolor=blue!20!white, linecolor=blue!20!white,
textsize=footnotesize]{todonotes}
\usepackage[colorlinks]{hyperref}

\newtheorem{thm}{Theorem}[section]
\newtheorem{lem}[thm]{Lemma}

\newtheorem{rem}[thm]{Remark}

\newcommand\cB{{\mathcal B}}
\newcommand\cC{{\mathcal C}}

\newcommand\cL{{\mathcal L}}

\newcommand\cP{{\mathcal P}}

\newcommand\bC{{\mathbb C}}

\newcommand\bN{{\mathbb N}}

\newcommand\bR{{\mathbb R}}

\newcommand\fkM{{\mathfrak M}}

\newcommand\ve{\varepsilon}

\newcommand\vf{\varphi}

\newcommand\Id{{\mathds{1}}}

\newcommand{\cPl}{\cP^{\text{long}}}
\newcommand{\cPs}{\cP^{\text{short}}}
\begin{document}

\title{A footnote on Expanding maps}
\author{Carlangelo Liverani}
\address{Carlangelo Liverani\\
Dipartimento di Matematica\\
II Universit\`{a} di Roma (Tor Vergata)\\
Via della Ricerca Scientifica, 00133 Roma, Italy.}
\email{{\tt liverani@mat.uniroma2.it}}
\thanks{ It is a pleasure to thank Oliver Butterley for pointing out the need of this note, for many interesting discussions related to these type of problems and for carefully reading a preliminary version of this note. I also thank Luigi Ambrosio for helpful references and Viviane Baladi and the anonymous referee for helpful comments. Work supported by the European Advanced Grant Macroscopic Laws
and Dynamical Systems (MALADY) (ERC AdG 246953).}
\begin{abstract}
I introduce Banach spaces on which it is possible to precisely characterize the spectrum of the transfer operator associated to a piecewise expanding map with H\"older weight.
\end{abstract}
\maketitle

\section{Introduction}
Lately there is some renewed interest in different norms which allow one to analyze transfer operators associated to expanding maps. Such an interest has several motivations, one of the most relevant being the study of semi flows arising from Lorenz like models (e.g., see \cite{AGP}). At the same time, the extension of transfer operator methods to the hyperbolic setting \cite{BKL, GL1, GL2,  BT1, BT2, BG1, BG2, DL, Li2, Ts, BL}, just to mention a few, has revitalized the subject. Particular attention has been devoted to the case in which the map is piecewise smooth and its derivative or the weight have low regularity (H\"older instead of $\cC^1$). Several possibilities have been and are currently being explored trying to improve on the classical BV scheme \cite{Ry, Co, Li1} or its relevant variants \cite{Ke, Sa}. Two recent interesting contributions are \cite{Th,Bu}.

The purpose of this note is to comment on an old proposal of mine, put forward in footnote 12 of \cite[page 193]{Li}, that has gone mostly unnoticed and/or not understood. Here I show that it can be easily applied to many relevant situations yielding the strongest results so far. I present the approach in the expanding one dimensional case but I see no obstacles in extending it to higher dimensions (following \cite{Li1}) or, with some more work, to the hyperbolic setting.

In the next section I will detail the proposed Banach space which is a weakening of $BV$ in the spirit of fractional order Sobolev spaces but avoiding completely definitions based on Fourier transforms. In the final section I will detail some settings where the above strategy can be applied and I will give the main result of the paper.

\section{ The Banach space}
Let $\fkM$ be the Banach space of complex valued Borel measures on $[0,1]$ equipped with the total variation norm.
For each $\vf\in \cC^1([0,1],\bC)$, $\mu\in \fkM$, let us define, for each $\alpha\in[0,1]$,
\[
\begin{split}
&|\vf|_\alpha=\sup_{x\in [0,1]}|\vf(x)|+\sup_{x,y\in[0,1]}\frac{|\vf(x)-\vf(y)|}{|x-y|^\alpha}\\
&\|\mu\|_\alpha=\sup_{\{\vf\in\cC^1\;:\;|\vf|_\alpha\leq 1\}}|\mu(\vf')|.
\end{split}
\]
We can then define $\cB_\alpha=\{\mu\in\fkM\;:\;\|\mu\|_\alpha<\infty\}$.  Note that $\cB_0$ is the space of absolutely continuos measures with density in $BV$ while $\cB_1=\fkM$.
\begin{rem} In the following we give, for the reader's convenience, a self contained proof of the relevant properties of the spaces $\cB_\alpha$. Note however that some results are known in larger generality than needed here \cite{Gr, Zu}. Also there is a connection between our spaces and the theory of BV functions on snowflake spaces \cite{Zu}.
\end{rem}
\begin{lem}
If $\alpha\in [0,1)$ and $\mu\in\cB_\alpha$, then $\mu$ is absolutely continuous with respect to Lebesgue and, calling $h$ its density,\footnote{ In this note the $L^p$ spaces are all w.r.t. the Lebesgue measure.} 
\[
|h|_{L^{\frac 1{\alpha}}}\leq 2 \|\mu\|_\alpha.
\] 
In addition $\cB_\alpha$ is a Banach space.
\end{lem}
\begin{proof}
Let $\vf\in \cC^0$ and $\mu\in \cB_\alpha$, and define $\phi(x)=\int_0^x\vf$. Then
\[
\mu( \vf)=\mu\left( \phi'\right).
\]
Since $\left|\phi(x)\right|\leq |\vf|_{L^{\frac1{1-\alpha}}}$ and
\[
|\phi(x)-\phi(y)|=\left|\int_x^y\vf\right|\leq |\vf|_{L^{\frac1{1-\alpha}}} |x-y|^\alpha,
\]
it follows that $|\phi|_\alpha\leq 2 |\vf|_{L^{\frac1{1-\alpha}}}$ hence
\[
|\mu(\vf)|\leq 2\|\mu\|_\alpha|\vf|_{L^{\frac 1{1-\alpha}}}.
\]
Thus $\mu$ belongs to the dual of $L^{\frac1{1-\alpha}}$, i.e. $L^{\frac 1{\alpha}}$, hence $\mu$ is absolutely continuous with respect to Lebesgue, let $h$ be the density. Then
\[
|h|_{L^{\frac 1\alpha}}=\sup_{|\vf|_{L^{\frac1{1-\alpha}}}\leq 1}\mu(\vf)\leq 2\|\mu\|_\alpha.
\] 
To verify that $\cB_\alpha$ is a Banach space it suffices to see that it is complete. Let $\{\mu_n\}\subset \cB_\alpha$ be a Cauchy sequence in $\cB_\alpha$ and $\{h_n\}$ be the respective densities. Then $\{h_n\}$ is a Cauchy sequence in $L^{\frac 1\alpha}$, let $h$ be its limit. Setting $\mu(\vf):=\int_0^1 h\vf$, for each $\vf\in\cC^1$,\footnote{ In this note I use $C_\#$ to designate a generic constant.}
\[
|\mu(\vf')|=\lim_{n\to\infty}|\mu_n(\vf')|\leq C_\#|\vf|_\alpha.
\]
Thus, $\mu\in\cB_\alpha$. On the other hand, for each $\ve>0$, there exists $n\in\bN$ such that $\|\mu_n-\mu_m\|_\alpha\leq \ve$ for all $m\geq n$. Then, for each $\vf\in\cC^1$ and $m>n$,
\[
|\mu_n(\vf')-\mu(\vf')|\leq \ve |\vf|_\alpha+|\mu_m(\vf')-\mu(\vf')|.
\]
Taking the limit for $m\to \infty$ it follows that $\mu$ is the limit of $\{\mu_n\}$ in $\cB_\alpha$.
\end{proof}

Given the above Lemma we can as well consider the space of densities equipped with the norm
\[
\|h\|_\alpha=\sup_{|\vf|_\alpha\leq 1}\left|\int_0^1 h\vf'\right|.
\]
By a little abuse of notations we will call such a Banach space $\cB_\alpha$ as well. Since $\cB_\alpha\subset L^1([0,1])$ it is then convenient to use $L^1$ rather than $\fkM$ as a weak space.\footnote{ Note that the norm is exactly the same.}

\begin{lem}\label{lem:compactness}
For each $\alpha\in (0,1)$ the unit ball of $\cB_\alpha$ is relatively compact in $L^1([0,1])$.
\end{lem}
\begin{proof}
To start, we need a little preliminary result. Since the functions in the set $\{\int_0^x\vf\}_{|\vf|_\infty\leq 1}$ are uniformly Lipschitz they are, by Ascoli-Arzel\'a, relatively compact in the $\alpha$-H\"older topology. Thus, for each $\ve>0$ there exists a set $S_\ve:=\{\phi_i\}_{i=1}^{n_\ve}$, $|\phi_i|_{\cC^1}<\infty$ and $\sup_i|\phi_i|_\alpha\leq 1$, such that, for each $\vf\in L^\infty$, $|\vf|_\infty\leq 1$, setting $\Phi(x)=\int_0^x \vf$, we have
\[
\inf_{i\in\{1,\dots,n_\ve\}}\left|\Phi-\phi_i\right|_\alpha\leq \ve.
\]
Accordingly, for each $h\in\cB_\alpha$,
\[
\int h\vf=\int h\Phi'\leq \ve\|h\|_\alpha+\sum_{i=1}^{n_\ve}\left|\int h\phi_i'\right|.
\]
Taking the sup on $\vf$ we have then
\[
|h|_{L^1}\leq \ve\|h\|_\alpha+\sum_{i=1}^{n_\ve}\left|\int h\phi_i'\right|.
\]

We are now ready to conclude the proof.
Since we deal with metric spaces it suffices to check sequential compactness. 
Let $\{h_n\}_{n\in\bN}\subset \{ h\in\cB_\alpha\;:\;\|h\|_\alpha\leq 1\}$.
Define $\overline S=\cup_{j\in\bN}S_{2^{-j}}$. Note that 
\[
\sup_{\substack{\phi\in \overline S}}\left|\int h_n \phi'\right|\leq 1.
\]
Hence, by Tychonoff Theorem, we can extract a subsequence $\{n_j\}$ such that $\int h_{n_j}\phi'$ is convergent, when $j\to\infty$, for all $\phi\in\overline S$. It follows that $\{h_{n_j}\}$ is Chauchy in $L^1$ since, for all $\ve>0$,
\[
|h_{n_j}-h_{n_k}|_{L^1}\leq 2\ve+\sum_{i=1}^{n_\ve}\left|\int [h_{n_j}-h_{n_k}]\phi_i'\right|<3\ve,
\]
where we have chosen $j,k$ large enough.
\end{proof}
\begin{lem}\label{lem:lip} For each $\alpha\in (0,1)$,  $h\in\cB_\alpha$ and $\vf$ Lipschitz we have\footnote{ Note that, by  Rademacher's Theorem, $\vf$ is almost surely differentiable with bounded derivative, hence the integral is meaningful.}
\[
\int_0^1 h\vf'\leq |\vf|_\alpha \|h\|_\alpha.
\]
\end{lem}
\begin{proof}
It is convenient to extend $h$ to be zero and $\vf$ to be continuous and constant outside $[0,1]$, so we can regard all the integral as integral on $\bR$. Let $j_\ve$ be a smooth mollifier, then\footnote{ As usual $j*h(x)=\int_{\bR}j(x-y)h(y)dy$.}
\[
 \int_0^1 h\vf'=\lim_{\ve\to 0}\int_{\bR}\left(j_\ve* h\right)\vf'=\lim_{\ve\to 0}\int_{\bR}h(j_\ve* \vf)'\leq \lim_{\ve\to 0}\|h\|_\alpha |j_\ve*\vf|_\alpha\leq \|h\|_\alpha |\vf|_\alpha.
 \]
\end{proof}

\section{Piecewise differentiable maps and H\"older continuous weights} 

Let $f$ be a (almost everywhere defined) map of the interval $[0,1]$ in itself and $\cP$ a (possibly infinite) collection of open subintervals of $[0,1]$. We assume that $\cup_{p\in\cP}\,p$ has full Lebesgue measure in $[0,1]$. Also, assume that $f\in\cC^1(p,\bR)$ and $f, \frac 1{f'}\in\cC^0(\bar p,\bR)$ for each $p\in\cP$. Moreover, we assume the map to be expansive
\[
\inf_{p\in\cP}\inf_{x\in p} |f'(x)|>1.
\]
Let $\xi:[0,1]\to\bC$ be a function that we will call the {\em weight}. We assume that there exits $\beta\in (0,1]$ such that, for each $p\in\cP$, $ \xi \in C^\beta(\bar p,\bR)$, with uniform $\beta$-H\"older constant.  In addition, we require that $|f'|^r\in L^1$, for some $r\geq 0$, $\xi \cdot f'\in L^\infty$ and that there exists $\gamma\in [0,1)$ such that\footnote{ Here and in the following $x^t$, $x, t\in\bR$, is meant as a complex number and $|\cdot|$ is used both for the absolute value and the complex modulus.}
\begin{equation}\label{eq:optimal?}
\sum_{p\in\cP}\sup_{z\in p}|\xi (z)f'(z)^\gamma|^{\frac 1{1-\gamma}}<\infty.
\end{equation}
Next, let $\cL_\xi$ be the transfer operator (see \cite{Ba} for the relevance of such operators) defined by
\[
\cL_\xi h(x)=\sum_{y\in f^{-1}(x)}\xi(y)h(y).
\]
The main result of this note is the following.
\begin{thm}\label{thm:main} In the above setting, if $\beta>\frac 1{r+1}$, then, for each 
\[
1>\alpha>\max\left\{\gamma, 1-\beta, \frac {1-\beta}{1-\beta(1-r)}\right\},
\]
the spectral radius of $\cL_\xi$, when acting on $\cB_\alpha$, is bounded from above by  $|\xi f'|_\infty$, while the essential spectral radius is bounded by $\min\{|\xi f'|_\infty,8|\xi (f')^{\alpha}|_\infty\}$.
\end{thm}
Before discussing the proof of the above result let us indulge in several remarks.
\begin{rem} The proof will not use anywhere the condition $|f'|>1$. Yet, notice that if such a condition is not satisfied, then the theorem is easily empty since the bound for the essential spectral radius might equals the bound for the spectral radius. Yet, a small generalization is possible, we leave it to the interested reader.
\end{rem}  

\begin{rem} Note that, as stated, for $\gamma=0, \beta=1$ the Theorem does not cover the case $\alpha=0$. This is the usual BV case and it is well known already. Also, in the case $\beta=0$ there is no reason to expect good spectral properties for $\cL_\xi$.
\end{rem}
\begin{rem} As usual, by applying Theorem \ref{thm:main} to a large power of $\cL_\xi$, much sharper estimates of the spectral and essential spectral radius can be obtained (in particular the $8$ in the Theorem is superficial, this is why I did not strive to improve it). I leave this exercise to the reader (see \cite{GL} for relevant results).
\end{rem}
\begin{rem} To compare the above result with the literature remark that \cite{Ke}, at least in the published version, applies only to the case $\xi=\frac 1{f'}$ and when the partition $\cP$ is finite. The results in \cite{Th} apply only to the case $f'\in L^\infty$. Finally, in \cite{Bu} it is assumed \eqref{eq:optimal?} with $\gamma=0$.  Note that 
\[
\sum_p|\xi (f')^\gamma |_{L^\infty(p)}^{\frac{1}{1-\gamma}}\leq |\xi f'|_\infty^{\frac{\gamma}{1-\gamma}}\sum_p|\xi |_{L^\infty(p)}\leq C_\#\sum_p|\xi |_{L^\infty(p)}
\]
thus the present condition is weaker. Also in \cite{Bu} it appears the condition $f'\in L^r$, $r\geq 1$ with $\beta>\frac 1r$ that here is replaced by $r\geq 0$, $\beta>\frac 1{r+1}$. In particular, if $\beta>\frac 12$, one can treat the case $r=1$, which is the natural condition when the partition is finite. In the case of infinite partitions $r=1$ is not natural anymore (think of the Gauss map), yet a $r<1$ may suffice. Note hoverer that, most likely, the above cited results can be improved with some extra work. In particular, the norms in \cite{Ke} could provide a bound in which no condition on $f'$ is required \cite{private} while \cite{Th} could probably be improved by using a partition of unity in the spirit of \cite{Ba1, BT1}. Even so, the present norms seem to be an interesting candidate for extensions to the hyperbolic setting.
\end{rem}
\begin{rem} \label{rem:infty}The reader should be advised that the goal of this note is not to treat the most general case but to show that the $\cB_\alpha$ spaces can be conveniently used to investigate a vast class of problems. The optimal conditions under which Theorem \ref{thm:main} holds depend heavily on the situation. The present treatment is specially adapted to the case of finite partitions with weights that can also be zero. In the case of infinite partitions it could be more natural to consider weight of the form $\xi=e^{\phi}$, where $\phi$ is called the {\em potential}, and impose the H\"older condition on the potential, see Theorem \ref{thm:main2}.\footnote{ If $\xi$ vanishes somewhere one can still use such a setting by introducing countably many {\em artificial} partition elements, in the spirit of billiards {\em homogeneity strips}.}
\end{rem}
Theorem \ref{thm:main} follows in a standard way (see \cite{Ba}) from Lemma \ref{lem:compactness} and the next Lasota-Yorke inequality.
\begin{lem}\label{LY}
If $\beta>\frac 1{r+1}$, then for each $1>\alpha>\max\{\gamma, 1-\beta, \frac{1-\beta}{1-\beta(1-r)}\}$, there exists $B>0$ such that, for all $h\in\cB_\alpha$,
\[
\begin{split}
&|\cL h|_{L^1}\leq |\xi f'|_\infty |h|_{L^1}\\
&\|\cL h\|_\alpha\leq 8|\xi\cdot (f')^{\alpha}|_\infty \|h\|_\alpha+ B |h|_{L^1}.
\end{split}
\]
\end{lem}
\begin{proof}
For each $h\in L^1, \vf\in L^\infty$ we have, by a change of variable on each $p\in\cP$,
\[
\int \cL_\xi h\cdot \vf =\int h\cdot \xi \cdot f'\cdot \vf\circ f,
\]
from which the first inequality of the Lemma readily follows.\newline
Let $h\in \cB_\alpha$. For each $\vf\in\cC^1$ such that $|\vf|_\alpha\leq 1$, we have
\[
\int \cL_\xi h\cdot \vf' =\sum_{p\in\cP}\int_p h \xi (\vf\circ f)'.
\]
First of all, we want to take care of the fact that $f'$ may blow up at the boundaries of $p\in\cP$ so that $\vf\circ f$ may fail to be Lipschitz on $\bar p$ (preventing us from using Lemma \ref{lem:lip}). At the same time we would like to approximate $\xi$ by more regular functions since during the computation we will need to take the derivative of the weigh and $\xi$ is only H\"older. A nice possibility is to use piecewise constant functions $\xi_k$, so that the problem of taking derivatives can be handled just by a refining of the partition $\cP$. This procedure must be done with some care since in the following it is essential to retain the property $\xi_k\cdot f'\in L^\infty$. 

Since, by hypothesis, $|\xi (\vf\circ f)'|_{\infty}\leq |\xi f'|_{\infty}\leq C_\#$, it follows that $\xi$ is zero where $f'$ blows up. For each $\ve>0$ we can then consider the functions $\bar\xi_{\ve}(z)=\max\{|\xi(z)|-\ve, 0\}$ and $\tilde\xi_\ve=\frac{\xi}{|\xi|}\cdot\bar \xi_\ve$. Clearly $\tilde\xi_\ve$ is zero in a neighborhood of the points in which $f'$ explodes, also $|\tilde \xi_\ve|\leq |\xi|$ and they have $\beta$-H\"older constant uniformly (in $\ve$) proportional. 
By Lebesgue Dominated Convergence Theorem, for each $h\in\cB_\alpha,\vf\in\cC^1$, $|\vf|_\alpha\leq 1$ there exists $\ve$ such that
\begin{equation}\label{eq:part0}
\left|\sum_{p\in\cP}\int_p h (\xi-\tilde\xi_\ve) (\vf\circ f)'\right|\leq |h|_{L^1}.
\end{equation}
Next, for each $k\in\bN$, let $\cP_k$ be a refinement of $\cP$ such that all the elements of $\cP$ of length larger than $2^{-k+1}$ are partitioned in elements of length between $2^{-k+1}$ and $2^{-k}$. Let $\cPl_k=\{p\in\cP_k \;:\; p\not\in\cP\}$ (this is the collection of elements that come from the refinement and hence are longer than $2^{-k}$, note that they are a finite number) and $\cPs_k=\{p\in\cP_k \;:\; p\in\cP\}$ (these are the shorter element). For each $p\in\cP_k$ let $x_p\in\bar p$ be such that $|\tilde\xi_{\ve}(x_p)|=\inf_{z\in p}|\tilde\xi_{\ve}(z)|$. For each $k\in\bN$ let $\xi_k(x)=\tilde\xi_\ve(x_p)$ for all $x\in p\in\cP_k$. By construction, $\xi_k\in L^\infty$ and
\[
\left|\xi_k-\tilde\xi_\ve\right|_\infty\leq C_\# 2^{-\beta k}.
\]
It is now convenient to define $\rho_k=\xi_{k+1}-\xi_{k}$. Note that, for all $k_0\in\bN$,
\begin{equation}\label{eq:rho}
\begin{split}
&\sum_{k\geq k_0}\rho_k=\tilde\xi_\ve-\xi_{k_0},\\
&|\rho_k|_\infty\leq C_\# 2^{-\beta k}.
\end{split}
\end{equation}
Hence,\footnote{ From now on we will write $(\vf\circ f)'$ for $\sum_{p\in\cP_{k}}\Id_p(\vf\circ f)'$, i.e. the derivative is meant in the strong sense but only where it is defined. By the way, given a set $A$, the indicator function $\Id_A$ is defined by $\Id_A(x)=1$ if $x\in A$ and zero otherwise.}
\begin{equation}\label{eq:ly-0}
\begin{split}
&\sum_{p\in\cP_{k_0}}\int_p h \tilde\xi_\ve(\vf\circ f)'=\sum_{k\geq k_0}\sum_{p\in\cP_{k_0}}\int_p h(\vf\circ f)'\rho_k+\sum_{p\in\cP_{k_0}}\int_p h (\vf\circ f)'\xi_{k_0}\\
&=\sum_{p\in\cP_{k_0}}\int_p h (\vf\circ f\cdot \xi_{k_0})'
+\sum_{k\geq k_0}\int_0^1 h\frac d{dx}\left[\int_0^x\sum_{p\in\cP_{k_0}}\Id_p(\vf\circ f)'\rho_k\right] ,
\end{split}
\end{equation}
where the convergence of the series on the first line follows because $f'$ is bounded on the support of $\tilde\xi_\ve$, an hence on the support of the $\rho_k$.
To continue, let $\tilde\ell$ be linear in each $p\in\cP^{\text{long}}_{k_0}$ and equal to $\vf\circ f\cdot \xi_{k_0}$ on $\partial p$.
Then we define
\[
\zeta_{k_0}(x)=\begin{cases}\vf\circ f\cdot \xi_{k_0}(x)-\tilde \ell(x)\quad&\forall x\in p\in\cP^{\text{long}}_{k_0}\\
0&\text{otherwise}.
\end{cases}
\]
Note that $\zeta_{k_0} \in \cC^0([0,1],\bR)$, in fact Lipschitz.  In addition, for $x,y\in p\in\cP$, 
\begin{equation}\label{eq:basic}
\left|\vf\circ f(x)-\vf\circ f(y)\right|\leq \left|\int_x^yf'(z)dz\right|^\alpha\leq |f'(w)|^\alpha|x-y|^\alpha
\end{equation}
for some $w\in [x,y]$.
Thus for each $x,y\in p\in\cPl_{k_0}$,
\[
\left|\zeta_{k_0}(x)-\zeta_{k_0}(y)\right|\leq 2|\xi (f')^\alpha|_\infty|x-y|^\alpha .
\]
On the other hand, if $x$ and $y$ belong to different elements of $p\in\cPl_{k_0}$, let $b_1, b_2\in [x,y]$ the boundaries of the elements to which $x$ and $y$ belong, respectively. Since, by construction, $\zeta_{k_0}=0$ at the boundaries of the elements of $\cP_{k_0}$, we have\footnote{ In the last line we use H\"older inequality: $\sum_i a_i b_i\leq \left[\sum_i a_i^\frac 1\alpha\right]^{\alpha}\left[\sum_i b_i^{\frac 1{1-\alpha}}\right]^{1-\alpha}$.}
\[
\begin{split}
\left|\zeta_{k_0}(x)-\zeta_{k_0}(y)\right|&\leq  2 |\xi(f')^\alpha|_\infty (|x-b_1|^\alpha+|y-b_2|^\alpha)\\
&\leq 2^{2-\alpha} |\xi(f')^\alpha|_\infty|x-y|^\alpha.
\end{split}
\]

Putting together the above facts we have
\begin{equation}\label{eq:first}
\begin{split}
&\left|\zeta_{k_0}\,\right|_\infty\leq 2 |\xi |_\infty\\
&\left|\zeta_{k_0}\,\right|_{\alpha}\leq 6  |\xi \cdot(f')^\alpha|_\infty.
\end{split}
\end{equation}
We can then write the first term of the second line of \eqref{eq:ly-0} as
\[
\begin{split}
&\sum_{p\in\cP_{k_0}}\int_p h (\vf\circ f\cdot \xi_{k_0})'=\int_0^1 h\zeta_{k_0}'+\sum_{p\in\cPl_{k_0}}\int_p h\tilde\ell'+\sum_{p\in\cPs_{k_0}}\int_p h (\vf\circ f\cdot \xi_{k_0})'\\
&\leq 6|\xi \cdot(f')^\alpha|_\infty\|h\|_\alpha+ C_{k_0} |h|_{L^1}+\int_0^1 h \frac d{dx}\int_0^x\sum_{p\in\cPs_{k_0}}\Id_p(\vf\circ f\cdot \xi_{k_0})'.
\end{split}
\]
Note that we cannot avoid the separation between the short and long pieces: if we would have defined $\tilde\ell$ as a linear interpolation on all the intervals, then $\sup_p|\tilde\ell'|_{\cC^0(\bar p,\bR)}$ could have been infinite. As is made clear clear by the above expression, to handle the short pieces we use a by now standard idea: we estimate using the strong norm. Thus we must compute the norm of the test function:
\[
\begin{split}
&\left|\int _x^y\sum_{p\in\cPs_{k_0}}\Id_p(\vf\circ f\cdot \xi_{k_0})'\right|\leq \sum_{\substack{p\in\cPs_{k_0}\\p\cap [x,y]\neq \emptyset}}\sup_{z\in p}|\xi (z) f'(z)^\alpha| |p\cap [x,y]|^\alpha\\
&\leq \left[\sum_{p\in\cPs_{k_0}}\sup_{z\in p}|\xi (z) f'(z)^\alpha|^{\frac 1{1-\alpha}}\right]^{1-\alpha}|x-y|^\alpha \\
&\leq C_\#|\xi f'|_\infty^{\frac{\alpha-\gamma}{1-\gamma}}\left[\sum_{\{p\in \cP\;:\; |p|\leq 2^{-k_0}\}}\sup_{z\in p}|\xi (z)f'(z)^\gamma|^{\frac{1}{1-\gamma}}\right]^{1-\alpha}\hskip-.5cm |x-y|^\alpha\leq \frac{|\xi \cdot(f')^\alpha|_\infty}2|x-y|^\alpha
\end{split}
\]
where we have used the hypothesis $\alpha\geq \gamma$, used condition \eqref{eq:optimal?} and chosen $k_0$ large enough (so that the tail of the convergent series is as small as needed).
Accordingly
\begin{equation}\label{eq:ly-1}
\sum_{p\in\cP_{k_0}}\int_p h (\vf\circ f\cdot \xi_{k_0})'\leq 7\,|\xi \cdot(f')^\alpha|_\infty\|h\|_\alpha+ C_{k_0} |h|_{L^1}.
\end{equation}
To estimate the second term in the second line of \eqref{eq:ly-0} note that, by construction, $\rho_k$ is zero on the elements $p\in \cPs_{k+1}$. We can then consider a piecewise linear approximation $\ell_k$ of $\vf\circ f$ constructed by taking linear pieces on each $p\in\cPl_{k+1}$ and such that the two functions are equal on $\partial\cPl_{k+1}$. Then, we define $\eta_k(x)=0$ if $x\in p\in \cPs_{k+1}$ and $\eta_k=\rho_k(\vf\circ f-\ell_k)$ otherwise. Note that $\eta_k\in\cC^0$ and Lipschitz.  We can write the test functions in the second term in the second line of \eqref{eq:ly-0} as
\begin{equation}\label{eq:psik}
\psi_k(x):=\int_0^x(\vf\circ f)'\rho_k=\int_0^x\eta_k'+ \int_0^x\ell_k'\rho_k.
\end{equation}
We are left with the task of estimating $|\psi_k|_\alpha$. We will treat the two terms separately.

To start, note that (by \eqref{eq:rho}) 
\[
\left|\int_x^y\eta_k'\right|=\left|\sum_{p\in\cPl_{k+1}}\int_{p\cap [x,y]}\eta_k'\right|\leq C_\#2^{-\beta k}
\]
since at most two of the elements of the sum are non zero, given that $\eta_k$ is zero at the boundaries of $\cP_{k+1}$. On the other hand if $p=[a,b]$ and $p\cap [x,y]=[a',b']$ we have\footnote{ To obtain the last line note that the second term in the curly bracket of the second line is bounded by $C_\#|p|^{\alpha-1}|p\cap [x,y]|$ and $|p\cap [x,y]|\leq |p|^{1-\alpha}|x-y|^\alpha$.}
\[
\begin{split}
\left|\int_{p\cap[x,y]}\eta_k'\right|&\leq |\rho_k(a)|\left \{|\vf(f(b'))-\vf(f(a'))|+\frac{|\vf(f(b))-\vf(f(b))|}{|p|}|p\cap [x,y]|\right\}\\
&\leq |\rho_k(a)|^{1-\alpha}2^\alpha\left \{\left[\int_{p\cap[x,y]}|\xi f'|\right]^\alpha+\left[\int_{p}|\xi f'|\right]^\alpha\frac{|p\cap [x,y]|}{|p|}\right\}\\
&\leq C_\# 2^{-\beta(1-\alpha) k} |x-y|^\alpha
\end{split}
\]
where we have used the hypothesis $\xi f'\in L^\infty$ and \eqref{eq:basic}. Accordingly,
\begin{equation}\label{eq:part-one}
\left|\int_0^{(\cdot)} \eta_k'\right|_\alpha\leq C_\# 2^{-\beta(1-\alpha) k}.
\end{equation}

Next, we must estimate the second term in \eqref{eq:psik}
\begin{equation}\label{eq:start-last}
\begin{split}
\left|\int_x^y\ell'_k\rho_k\right|&\leq \sum_{\substack{p\in\cPl_{k+1}\\p\cap[x,y]\neq\emptyset}}|\rho_k|_{L^\infty(p)}\left|\int_pf'\right|^\alpha\frac{|p\cap[x,y]|}{|p|}\\
&\leq C_\# 2^{-\beta(1-\alpha) k}\sum_{\substack{p\in\cPl_{k+1}\\p\cap[x,y]\neq\emptyset}}\left|\int_p|\xi f'|\right|^\alpha\frac{|p\cap[x,y]|}{|p|}\\
&\leq C_\# 2^{-\beta(1-\alpha) k+(1-\alpha)k}|x-y|\leq C_\# 2^{-\epsilon(1-\alpha) k} |x-y|^\alpha,
\end{split}
\end{equation}
provided
\begin{equation}\label{eq:xy-small}
|x-y|\leq 2^{-[\epsilon +(1-\beta)]k}.
\end{equation}
To treat the $x,y$ for which \eqref{eq:xy-small} fails we estimate differently the first line of \eqref{eq:start-last}. To do so it is convenient to divide the discussion in two case. If $r\leq 1$, then
\begin{equation}\label{eq:almost1}
\begin{split}
\left|\int_x^y\ell'_k\rho_k\right|&\leq \sum_{\substack{p\in\cPl_{k+1}\\p\cap[x,y]\neq\emptyset}}|\rho_k|_{L^\infty(p)}^{1-\alpha(1-r)}\left|\int_p|\xi|^{1-r}f'\right|^\alpha\frac{|p\cap[x,y]|^{1-\alpha}}{|p|^{1-\alpha}}\\
&\leq C_\#2^{-\beta(1-\alpha(1-r))k+(1-\alpha)k}\sum_{p\in\cP_{k+1}}\left|\int_p |f'|^r\right|^\alpha |p\cap[x,y]|^{1-\alpha}\\
&\leq C_\#2^{-\beta(1-\alpha(1-r))k+(1-\alpha)k}\left|\int_0^1 |f'|^r\right|^\alpha |y-x|^{1-\alpha}.
\end{split}
\end{equation}
If $\alpha\leq \frac 12$, then, for some $\epsilon>0$,
\begin{equation}\label{eq:part2}
\left|\int_x^y\ell'_k\rho_k\right|\leq C_\# 2^{-\epsilon k}|y-x|^{\alpha}.
\end{equation}
provided $\alpha>\frac{1-\beta}{1-\beta(1-r)}$. 
If, instead, $\alpha >\frac 12$ then we can continue the estimate in \eqref{eq:almost1} by using \eqref{eq:xy-small} to yield
\[
\left|\int_x^y\ell'_k\rho_k\right|\leq C_\# 2^{-\beta(1-\alpha(1-r))k+(1-\alpha)k} 2^{[\epsilon+(1-\beta)](2\alpha-1)k}|x-y|^\alpha\leq C_\# 2^{-\epsilon k}|x-y|^\alpha
\]
provided $\beta>\frac {1}{1+r}$ and $\epsilon$ has been chose small enough.

On the other hand, if $r>1$, we have
\[
\begin{split}
\left|\int_x^y\ell'_k\rho_k\right|&\leq \sum_{\substack{p\in\cPl_{k+1}\\p\cap[x,y]\neq\emptyset}}|\rho_k|_{L^\infty(p)}\left|\int_p|f'|^r\right|^{\frac{\alpha}r}|p\cap[x,y]|\, |p|^{\alpha(1-\frac 1r)-1}\\
&\leq\sum_{\substack{p\in\cPl_{k+1}\\p\cap[x,y]\neq\emptyset}}2^{-\beta k+(1-\alpha)k}\left|\int_p|f'|^r\right|^{\frac{\alpha}r}|p\cap[x,y]|^{1-\frac\alpha r}\\
&\leq C_\# 2^{-\beta k+(1-\alpha)k}|x-y|^{1-\frac\alpha r}\leq C_\#2^{-\epsilon k}|x-y|^\alpha,
\end{split}
\]
provided 
\[
\frac r{1+r}\geq \alpha>1-\beta.
\]
While if $\alpha>\frac r{1+r}$, then
\begin{equation}\label{eq:part3}
\left|\int_x^y\ell'_k\rho_k\right|\leq C_\# 2^{-\beta k+(1-\alpha)k}2^{(\alpha+\frac{\alpha}r-1)[(1-\beta)+\epsilon]k} |x-y|^{\alpha}\leq 2^{-\epsilon k}|x-y|^{\alpha}
\end{equation}
provided $\beta>\frac 1{1+r}$ and $\epsilon$ has been chosen small enough.
Collecting equations \eqref{eq:part-one}, \eqref{eq:part2} and \eqref{eq:part3} we have that, for $\beta>\frac 1{1+r}$ and $\alpha>\frac{1-\beta}{1-\beta(1-r)}$, 
\[
|\psi_k|_\alpha\leq C_\#2^{-\epsilon k}.
\]
Thus, by choosing $k_0$ large enough, we have
\begin{equation}\label{eq:ly-2}
\sum_{k\geq k_0}\int_0^1 h\frac d{dx}\left[\int_0^x\sum_{p\in\cP_{k_0}}\Id_p(\vf\circ f)'\rho_k\right] \leq |\xi \cdot(f')^\alpha|_\infty\|h\|_\alpha.
\end{equation}
Finally, collecting \eqref{eq:part0}, \eqref{eq:ly-0},  \eqref{eq:ly-1},  \eqref{eq:ly-3}, we have
\begin{equation}\label{eq:ly-3}
\int_0^1 h \xi(\vf\circ f)'\leq 8\,|\xi \cdot(f')^\alpha|_\infty\|h\|_\alpha+ C_\#|h|_{L^1},
\end{equation}
from which the Lemma follows.
\end{proof}

We conclude with an alternative result, in order to give a taste of the available possibilities mentioned in Remark \ref{rem:infty}.
\begin{thm}\label{thm:main2}
If the potential is uniformly $\beta$-H\"older on the elements of the partition (see Remark \ref{rem:infty}), then we do not need to impose any condition on the integrability of $f'$ and Theorem \ref{thm:main} holds under the single condition $1>\alpha>\max\{\gamma,1-\beta\}$.
\end{thm}
\begin{proof}
We just need to prove Lemma \ref{LY} under the new condition. All the previous arguments are valid, the only difference is that now we do not need to introduce $\tilde\xi_\ve$ since $\xi$ is bounded away from zero on the elements of the partition and hence the derivative cannot explode. The approximation scheme yields 
\[
|\rho_k(x)|\leq C_\# 2^{-\beta k}|\xi(x)|
\]
rather than $|\rho_k|\leq C_\# 2^{-\beta k}$ as before. Accordingly, we can easily conclude the proof of Lemma \ref{LY} as follows.
\[
\left|\int_x^y\ell_k' \rho_k\right|\leq C_\# \sum_{p\in\cPl}\int_{p\cap[x,y]}2^{-\beta k}\frac{\left(\int_p |f'\xi|\right)^\alpha}{|p|}\leq C_\# 2^{-(\alpha+\beta-1)k}|x-y|.
\]
\end{proof}


\begin{thebibliography}{999}
\footnotesize

\bibitem{AGP} V. Araujo, S. Galatolo, M.-J. Pacifico, {\em Decay of correlations for maps with uniformly contracting fibers and logarithm law for singular hyperbolic attractors}, Preprint Arxiv  	arXiv:1204.0703.
\bibitem{Ba} V. Baladi, {\em Positive Transfer Operators \& Decay of Correlation}. Volume 16 of Advanced Series
in Nonlinear Dynamics. World Scientific, Singapore, 2000.
\bibitem{Ba1} V. Baladi, {\em Anisotropic Sobolev spaces and dynamical transfer operators: $C^\infty$ foliations.} Algebraic and topological dynamics, 123Ð135, Contemp. Math., {\bf 385}, Amer. Math. Soc., Providence, RI, 2005. 
\bibitem{BG1} V. Baladi and S. Gou\"{e}zel,
{\it Good Banach spaces for piecewise hyperbolic maps via interpolation,}
Ann. Inst. Henri Poincar\'e, Anal. non. lin. 
{\bf 26}  (2009) 1453--1481.
\bibitem{BG2} V. Baladi  and S. Gou\"{e}zel, {\it Banach
spaces for  piecewise cone hyperbolic maps,} J. Modern Dynamics  {\bf 4} (2010) 91--137.
\bibitem{BL} V. Baladi  and C. Liverani, {\em Exponential decay of correlations for piecewise cone hyperbolic contact flows}, Communications in Mathematical Physics, {\bf 314}, Number 3, (2012) 689--773
\bibitem{BT1} V. Baladi and M. Tsujii,
{\it Anisotropic H\"{o}lder and Sobolev spaces for hyperbolic
diffeomorphisms,} Annales de l'Institut Fourier {\bf 57} (2007)
127--154.
\bibitem{BT2} V. Baladi and M. Tsujii,
{\it Dynamical determinants and spectrum for hyperbolic diffeomorphisms,}
in: Probabilistic and Geometric Structures in Dynamics, K. Burns, D. Dolgopyat, 
\& Ya. Pesin (eds),  Contemp. Math., Amer. Math. Soc., {\bf 469} (2008) 29--68.
\bibitem{BKL} M. Blank, G. Keller, and C. Liverani,
{\it Ruelle-Perron-Frobenius spectrum for Anosov maps,}
Nonlinearity {\bf 15} (2002) 1905--1973.
\bibitem{Bu} O. Butterley, {\em An alternative approach to generalised BV and the application to expanding interval maps}, to appear in Discrete
Contin. Dyn. Syst.
\bibitem{private} O. Butterey, private communication.
\bibitem{Co} W. Cowieson, {\em Absolutely continuous invariant measures for most piecewise smooth expanding maps}. Ergodic Theory Dynam. Systems {\bf 22} (2002), no. 4, 1061--1078. 
\bibitem{DL} M. Demers and C. Liverani,
{\it Stability of statistical properties in two-dimensional
piecewise hyperbolic maps,} Trans. Amer. Math. Soc. {\bf 360} (2008) 4777--4814.
\bibitem{GL} V.M. Gundlach and Y. Latushkin,  {\em A sharp formula for the essential spectral radius of the Ruelle transfer operator on smooth and Hšlder spaces}. Ergodic Theory Dynam. Systems {\bf 23} (2003), no. 1, 175Ð191.
\bibitem{GL1} S. Gou\"{e}zel and C. Liverani,
{\em Banach spaces adapted to Anosov systems,}
Ergodic Theory Dynam. Systems {\bf 26} (2006) 189--218.
\bibitem{GL2} S. Gou\"{e}zel and C. Liverani,
{\it Compact locally maximal hyperbolic sets for smooth maps: fine statistical properties,}
J. Diff. Geom. {\bf 79} (2008) 433--477.
\bibitem{Gr} M. Gromov, {\em Metric structures for Riemannian and non-Riemannian spaces}, vol. {\bf 152} of Progress in Mathematics, Birkh\"auser Boston Inc., Boston, MA, 1999. With appendices by M. Katz, P. Pansu and S. Semmes.
\bibitem{Ke} G. Keller, {\em Generalized bounded variation and applications to piecewise monotonic transformations}.
Probability Theory and Related Fields,{ \bf 69}, no. 3, (1985) 461--478.
\bibitem{KL} G. Keller, C. Liverani, {\em
Stability of the spectrum for transfer operators. }
Ann. Scuola Norm. Sup. Pisa Cl. Sci. (4) 28 (1999), no. 1, 141--152. 
\bibitem{Li} C. Liverani, {\em Invariant measures and their properties. A functional analytic point of view.} Dynamical systems. Part II, 185--237, Pubbl. Cent. Ric. Mat. Ennio Giorgi, Scuola Norm. Sup., Pisa, 2003.
\bibitem{Li1} C. Liverani, {\em Multidimensional expanding maps with singularities: a pedestrian approach}. Ergodic Theory and Dynamical Systems. DOI:10.1017/S0143385711000939. Preprint arXiv:1110.2001v1.
\bibitem{Li2} C. Liverani, {\em On contact Anosov flows}. Ann. of Math. (2) {\bf 159} (2004), no. 3, 1275--1312.
\bibitem{Sa} B. Saussol, {\em Absolutely continuous invariant measures for multidimensional expanding maps}. Israel J. Math. {\bf 116} (2000), 223--248.
\bibitem{Ry} M. Rychlik,  {\em Bounded variation and invariant measures}.  Studia Math.  {\bf 76}  (1983),  no. 1, 69--80.
\bibitem{Th} D. Thomine, {A spectral gap for transfer operators of piecewise expanding maps}. Discrete
Contin. Dyn. Syst., {\bf 30}, no. 3 (2011), 917--944.
\bibitem{Ts} T. Masato {\em Quasi-compactness of transfer operators for contact Anosov flows}. Nonlinearity 23 (2010), no. {\bf 7}, 1495--1545.
\bibitem{Zu} Z. Roger {\em Integration of H\"older forms and currents in snowflake spaces}. Calc. Var. Partial Differential Equations {\bf 40} (2011), no. 1-2, 99--124.
\end{thebibliography}
\end{document}